\documentclass[12pt]{amsart}
\usepackage{amssymb,latexsym,amsmath,amsthm,amsfonts,amscd,amsbsy}
\usepackage{enumerate}

\makeatletter
\@namedef{subjclassname@2010}{%
  \textup{2010} Mathematics Subject Classification}
\makeatother

\newtheorem{thm}{Theorem}[section]

\newtheorem{lem}[thm]{Lemma}
\newtheorem{prob}[thm]{Problem}

\theoremstyle{definition}

\numberwithin{equation}{section}

\begin{document}


\baselineskip=17pt



\title[On $C(K,X)$ spaces containing   a complemented copy of $c_{0}(\Gamma)$]{When does $C(K,X)$  contain  a complemented copy of $c_{0}(\Gamma)$   iff  $X$ does?}
      
\author{Vin\'icius Morelli Cortes}
\address{University of S\~ao Paulo, Department of Mathematics, IME, Rua do Mat\~ao 1010,  S\~ao Paulo, Brazil}
\email{vinicius.cortes@usp.br}

\author{El\'oi Medina Galego}
\address{University of S\~ao Paulo, Department of Mathematics, IME, Rua do Mat\~ao 1010,  S\~ao Paulo, Brazil}
\email{eloi@ime.usp.br}

\subjclass{Primary 46B03, 46E15; Secondary 46E40, 46B25}

\keywords{Complemented subspaces, $c_0(\Gamma)$ spaces, $C(K, X)$ spaces}

\begin{abstract} Let $K$  be  a compact Hausdorff  space with weight w$(K)$, $\tau$  an infinite  cardinal with cofinality cf$(\tau)$ and  $X$  a Banach space.  In contrast with a   classical theorem of Cembranos  and Freniche it is shown that if  cf$(\tau)>$ w$(K)$ then the space $C(K, X)$ contains a complemented copy of $c_{0}(\tau)$ if and only if $X$ does.  

This result is optimal  for every infinite cardinal $\tau$, 
 in the sense that it can not be improved by replacing the inequality     cf$(\tau)>$ w$(K)$ by another weaker than it.

\end{abstract}

\maketitle




\section{Introduction and the main theorem}

Let $K$ be a compact Hausdorff space and $X$ a real Banach space. By $C(K,X)$ we denote the Banach space of all continuous $X$-valued functions defined on $K$ and endowed with the supremum norm. This space will be denoted by $C(K)$ in the case  where $X$ is the scalar field.

For a non-empty set $\Gamma$, $c_0(\Gamma)$ denotes the Banach space of all real-valued maps $f$ on $\Gamma$ with the property that for each $\varepsilon > 0$, the set $\{ \gamma \in \Gamma : |f(\gamma)| \geq \varepsilon\}$ is finite, equipped with the supremum norm. We will refer to $c_0(\Gamma)$ as $c_0(\tau)$ when the cardinality of $\Gamma$ (denoted by $|\Gamma|$) is equal to $\tau$. This space will be denoted by $c_0$ when $\tau = \aleph_{0}$. As usual,  the dual space of $c_0$ will be denoted by $\ell_1$.

We use standard set-theoretical and  Banach space theory terminology as may be found, e.g., in \cite{J} and \cite{JL} respectively.  If $X$ and $Y$ are Banach spaces, we say that $Y$ contains a  complemented copy of $X$ (in short, $X \stackrel{c}{\hookrightarrow} Y$) if $X$ is isomorphic to a complemented  subspace of $Y$. 

In 1982, Saab and Saab \cite{Sa} solved the problem of characterizing when $C(K, X)$ contains a complemented copy of $l_{1}$. More precisely, it was shown that if $K$ is a compact Hausdorff space and $X$ is a Banach space, then
\begin{equation}\label{saabl1}
\ell_1 \stackrel{c}{\hookrightarrow} C(K,X) \iff \ell_1 \stackrel{c}{\hookrightarrow} X.
\end{equation}
On the other hand, in 1984 Cembranos \cite{Ce} and independently Freniche \cite{F} showed that the statement analogous to \eqref{saabl1} for $c_0$ is far from being true. Indeed, they proved that for any infinite compact Hausdorff space $K$ and infinite dimensional Banach space $X$ we have 
\begin{equation} \label{CF} 
c_0 \stackrel{c}{\hookrightarrow} C(K,X). 
\end{equation}
 However,  in opposition to  (\ref{CF}) 
it is well-known, see e.g. \cite[Lemma 2]{O}, that  for all \emph{finite} compact spaces $K$ and  Banach spaces $X$ the following quite similar  result  to \eqref{saabl1} holds:
\begin{equation}\label{c0emC(K,X)}
c_0 \stackrel{c}{\hookrightarrow} C(K,X) \iff c_0 \stackrel{c}{\hookrightarrow} X.
\end{equation}
These facts raise naturally the  following question:
\begin{prob}\label{problema1}
Let $K$ be a compact Hausdorff space, $X$ a Banach space and $\tau$ an infinite cardinal. Under which conditions we have
$$c_0(\tau) \stackrel{c}{\hookrightarrow} C(K,X) \iff c_0(\tau) \stackrel{c}{\hookrightarrow} X?$$
\end{prob}
The purpose of this paper is to give a solution to Problem \ref{problema1} in terms of the cofinality of  $\tau$ and the weight of $K$. Our result includes that mentioned in (\ref{c0emC(K,X)}). Recall that if  $\tau$ is an infinite cardinal then  the \emph{cofinality} of $\tau$, denoted by cf$(\tau)$,  is the least cardinal $\alpha$ such that there exists a family of ordinals $\{\beta_i : i \in \alpha\}$ satisfying $\beta_i < \tau$ for all $i \in \alpha$, and $\sup\{\beta_i : i \in \alpha\} = \tau$.  A cardinal $\tau$ is said to be \emph{regular} when cf$(\tau)=\tau$; otherwise, it is said to be \emph{singular}. Moreover, if  $S$ is a topological space then the \emph{weight} of $S$,  denoted by w$(S)$,  is the least cardinality of an open topological  basis of $S$.

\

Our main result is as follows.

\begin{thm}\label{teoprincipal}
Let $K$ be a compact Hausdorff space, $X$  a Banach space and $\tau$  an infinite cardinal. If $\emph{cf}(\tau) > \emph{w}(K)$ then
$$ c_0(\tau) \stackrel{c}{\hookrightarrow} C(K,X) \iff c_0(\tau) \stackrel{c}{\hookrightarrow} X.$$
\end{thm}

 In order to prove Theorem \ref{teoprincipal},  we first solve  in section 2 the problem of characterizing when a Banach space $X$  contains a complemented copy of $c_{0}(\tau)$ for arbitrary infinite cardinal $\tau$ (Theorem \ref{schlumprechtgeneralizado}). { In section 3, we present the proof of  Theorem  \ref{teoprincipal} and in section 4 we show that the hypothesis cf$(\tau) > \text{w}(K)$ is sharp in Theorem \ref{teoprincipal} for every infinite cardinal $\tau$.

\section{Banach spaces $X$ containing a complemented copies of $c_0(\tau)$}\label{sec2}

We start by recalling that a family $(x^*_i)_{i \in \tau}$ in the dual space $X^*$ is said to be \emph{weak$^*$-null} if for each $x \in X$ we have
$$(x^*_i(x))_{i \in \tau} \in c_0(\tau).$$
 We will denote by $(e_i)_{i \in \tau}$ the unit-vector basis of $c_0(\tau)$, that is, $e_i(i)= 1$ and $e_i(j) = 0$ for each $i, j \in \tau$, $i \neq j$. If $\Gamma$ is a subset of $\tau$, we identify $c_0(\Gamma)$ with the closed subspace of $c_0(\tau)$ consisting of the maps $g$ on $\tau$ such that $g(\gamma) = 0$ for each $\gamma \in \tau \setminus \Gamma$.

In his PhD thesis, Schlumprecht \cite{S} obtained the following useful characterization of Banach spaces containing complemented copies of $c_0$.

\begin{thm}(\cite[Theorem 1.1.2]{C}) \label{schlumprecht}
Let $X$ be a Banach space. Then $X$ contains a complemented copy of $c_0$ if, and only if, there exist a basic sequence $(x_n)_{n \in \mathbb{N}}$ in $X$ that is equivalent to the unit-vector basis of $c_0$ and a weak$^*$-null sequence $(x^*_n)_{n \in \mathbb{N}}$ such that
$$\inf_{n \in \mathbb{N}} |x^*_n(x_n)| > 0.$$
\end{thm}

Our goal in this section is to extend this result to the  $c_{0}(\tau)$ spaces, where $\tau \geq \aleph_{0}$. We next state two auxiliary results.
\begin{thm} (\cite[Remark following Theorem 3.4]{R}) \label{rosenthal}
Let $X$ be a Banach space and $\tau$ an infinite cardinal. Let $T : c_0(\tau) \to X$ be a bounded linear operator  such that $\inf\{ \|T(e_i)\| : i \in \tau\} > 0$. Then there exists a subset $\tau' \subset \tau$ such that $|\tau'|=\tau$ and $T_{|c_0(\tau')}$ is an isomorphism onto its image.
\end{thm}

\begin{thm} (\cite[Corollary 2]{G}) \label{granero}
Let $\tau \neq \varnothing$ be a cardinal and $Y$ be a closed subspace of $c_0(\tau)$. Then $Y$ is complemented in $c_0(\tau)$ if and only if there exists a subset $\tau' \subset \tau$ such that $Y$ is isomorphic to $c_0(\tau')$.
\end{thm}

We are now ready to  prove the following  extension of Theorem \ref{schlumprecht}.

\begin{thm}\label{schlumprechtgeneralizado}
Let $X$ be a Banach space and $\tau$ be an infinite cardinal. The following are equivalent:
\begin{enumerate}
\item $X$ contains a complemented copy of $c_0(\tau)$.
\item There exist a family $(x_i)_{i \in \tau}$ equivalent to the unit-vector basis of $c_0(\tau)$ in $X$ and a weak$^*$-null family $(x^*_i)_{i \in \tau}$ in $X^*$ such that, for each $i, j \in \tau$,
$$x^*_i(x_j) = \delta_{ij}.$$
\item There exist a family $(x_i)_{i \in \tau}$ equivalent to the unit-vector basis of $c_0(\tau)$ in $X$ and a weak$^*$-null family $(x^*_i)_{i \in \tau}$ in $X^*$ such that
$$\inf_{i \in \tau} |x^*_i(x_i)| > 0.$$
\end{enumerate}
\end{thm}
\begin{proof}
$(1) \implies (2).$ By hypothesis, there exist a family $(x_i)_{i \in \tau}$ equivalent to the unit-vector basis of $c_0(\tau)$ in $X$ and $P : X \to Y$ a bounded linear projection onto $Y = \overline{\text{span}}\{x_i : i \in \tau\}$. Let $(\varphi_i)_{i \in \tau}$ be the family of coordinate functionals associated to $(x_i)_{i \in \tau}$ in $Y^*$. For each $i \in \tau$, consider $x^*_i = \varphi_i \circ P \in X^*$. Therefore, for each $i, j \in \tau$ we have
$$x^*_i(x_j) = \varphi_i(P(x_j))=\varphi_i(x_j)= \delta_{ij}.$$
Moreover, given $x \in X$, let $(\alpha_i)_{i \in \tau} \in c_0(\tau)$ be such that $P(x) = \sum_{i \in \tau} \alpha_ix_i$. Then we have
$$x^*_i(x) = \varphi_i( P(x)) = \alpha_i, \forall i \in \tau,$$
that is, $(x^*_i(x))_{i \in \tau} = (\alpha_i)_{i \in \tau} \in c_0(\tau)$. Thus, $(x^*_i)_{i \in \tau}$ is weak$^*$-null. This proves $(1) \implies (2).$

$(2) \implies (3).$ This is trivial.

$(3) \implies (1).$ Let $(x_i)_{i \in \tau}$ in $X$ and $(x^*_i)_{i \in \tau}$ in $X^*$ be as stated. Consider the linear operator $T : X \to c_0(\tau)$ defined by
$$T(x) = (x^*_i(x))_{i \in \tau} \in c_0(\tau), \forall x \in X.$$
Since
$$\sup_{i \in \tau} |x^*_i(x)| = \|T(x)\|_\infty < +\infty, \forall x \in X,$$
by the Uniform Boundedness Principle we have
$$M = \sup_{i \in \tau} \|x^*_i\| < +\infty.$$
Hence,
$$\|T(x)\|_\infty = \sup_{i \in \tau} |x^*_i(x)| \leq M\|x\|, \forall x \in X.$$
Thus, $T$ is bounded.

Let $S : c_0(\tau) \to Y$ be an isomorphism from $c_0(\tau)$ onto $Y = \overline{\text{span}}\{x_i : i \in \tau\}$ such that $S(e_i)=x_i$, for each $i \in \tau$. Then we have
$$\|(T\circ S)(e_i)\|_\infty = \|T(x_i)\|_\infty = \sup_{j \in \tau} |x^*_j(x_i)| \geq |x^*_i(x_i)| \geq \delta,$$
for each $i \in \tau$, where $\delta = \inf_{j \in \tau} |x^*_j(x_j)| > 0$. Therefore, by Theorem \ref{rosenthal}, there exists $\tau' \subset \tau$ such that $|\tau'| = \tau$ and $T \circ S|_{c_0(\tau')}$ is an isomorphism from $c_0(\tau')$ onto $Z = (T \circ S)(c_0(\tau'))$. Setting
$$Y'= \overline{\text{span}}\{x_i : i \in \tau'\} = S(c_0(\tau')),$$
we obtain that
$$T|_{Y'} =(T\circ S|_{c_0(\tau')}) \circ (S|_{c_0(\tau')})^{-1} : Y' \to Z$$
is an isomorphism from $Y'$ onto $Z$. Hence, $Z \subset c_0(\tau)$ is isomorphic to $c_0(\tau')$, and thus, by Theorem \ref{granero}, there exists $P : c_0(\tau) \to Z$ a bounded linear projection onto $Z$. Define
$$Q = (T|_{Y'})^{-1} \circ P \circ T : X \to Y'.$$
Then $Q$ is a bounded linear operator and
$$Q(y) = ((T|_{Y'})^{-1} \circ P \circ T)(y)= (T|_{Y'})^{-1}(T(y))= y, \forall y \in Y'.$$
Thus, $Q$ is a bounded linear projection from $X$ onto $Y' = S(c_0(\tau'))$. Since $Y'$ is isomorphic to $c_0(\tau')$ and $|\tau'| = \tau$, the proof is complete.
\end{proof}

\section{On $C(K,X)$ spaces containing a complemented copy of $c_0(\tau)$}

In this section we will prove Theorem \ref{teoprincipal}. Before proving this theorem, we need  to state two   lemmas.

\begin{lem}\label{lemma1}
Let $K$ be a compact Hausdorff space, $D \subset K$ be a dense subset of $K$, $\mathcal{C} \subset C(K)$ be a dense subset of $C(K)$, $X$ be a Banach space and $f \in C(K,X)$. For each $\varepsilon > 0$ there exist $m \geq 1$, $f_1,\ldots,f_m \in \mathcal{C}$ and $d_1,\ldots,d_m \in D$ such that
$$\|f- h_\varepsilon\|_\infty < \varepsilon,$$
where 
$$h_\varepsilon = \sum_{j=1}^m f_j(\cdot) f(d_j).$$
\end{lem}
\begin{proof}
Fix $\varepsilon > 0$. For each $t \in K$, let $U_t = \{ s \in K : \|f(s) - f(t)\| < \varepsilon/4\}$.
Since $\{U_t : t \in K\}$ is an open cover of $K$, there exist $t_1, \ldots, t_m \in K$ such that
$$K = U_{t_1} \cup \ldots \cup U_{t_m}.$$
Let $\{g_1, \ldots, g_m\} \subset C(K)$ be a partition of unity subordinate to the open cover $\{U_{t_1}, \ldots, U_{t_m}\}$. For each $j \in \{1,\ldots,m\}$, choose $d_j \in U_{t_j} \cap D$ and $f_j \in \mathcal{C}$ such that
$$\|f_j - g_j\|_\infty < \frac{\varepsilon}{4m (\|f\|_\infty + 1)}.$$
Define
$$h_\varepsilon = \sum_{j=1}^m f_j(\cdot) f(d_j).$$
We will show that $\|f-h_\varepsilon\|_\infty < \varepsilon$. Fix $t \in K$ and consider 
$$I_t = \{ j \in \{1, \ldots, m\} : t \in U_{t_j}\} \neq \varnothing.$$
Observe that if $j \in I_t$, then
$$\|f(t) - f(d_j)\| \leq \|f(t)- f(t_j)\| + \|f(t_j) - f(d_j)\| < \frac{\varepsilon}{2}$$
and if $i \in \{1,\ldots,m\} \setminus I_t$, then $g_i(t)=0$. Therefore we have
\begin{align*}
\|f(t)-h_\varepsilon(t)\| &\leq \left\| f(t) - \sum_{j=1}^m g_j(t)f(d_j) \right\| + \left\|\sum_{j=1}^m g_j(t)f(d_j) - h_\varepsilon\right\|\\
&= \left\| \sum_{j \in I_t}g_j(t)(f(t)-f(d_j)) \right\| + \left\|\sum_{j=1}^m (g_j(t)-f_j(t))f(d_j)\right\|\\
&\leq \sum_{j \in I_t} g_j(t) \|f(t)-f(d_j)\| + \sum_{j=1}^m |g_j(t)-f_j(t)| \|f(d_j)\|\\
&< \frac{3\varepsilon}{4}.
\end{align*}
Consequently, $\|f-h_\varepsilon\|_\infty < \varepsilon$.
\end{proof}

\begin{lem}\label{lemma2}
Let $I$ be an infinite set and $J$ be a non-empty set. Let $\{I_j\}_{j \in J}$ be a family of subsets of $I$ such that $\bigcup_{j \in J} I_j = I$. If \emph{cf}$(|I|) > |J|$, then there exists $j_0 \in J$ such that $|I_{j_0}| = |I|$.
\end{lem}
\begin{proof}
Suppose that the conclusion does not hold. Then we have $|I_j| < |I|$ for each $i \in J$ and thus, by the definition of cofinality,
$$\sup\{ |I_j| : j \in J\} < |I|.$$
Since $|J| < \text{cf}(|I|) \leq |I|$, we obtain
$$|I| = \left| \bigcup_{j \in J} I_j\right| \leq \max(|J|, \ \sup\{ |I_j| : j \in J\}) < |I|,$$
a contradiction that finishes the proof.
\end{proof}

We are now in a position to prove our main result. By rcabv$(K,X)$ we denote the Banach space of all regular, countably additive $X$-valued Borel measures on $K$ with bounded variation, endowed with the variation norm. We  also recall that the \emph{density character} of  a topological space $S$,  denoted by dens$(S)$, is the least cardinality of a dense subset of $S$.

\

{\bf{Proof of Theorem \ref{teoprincipal}}}. We will show the non-trivial implication. We distingush two cases.

Case 1: $K$ is infinite. By Theorem \ref{schlumprechtgeneralizado}, there exist $(f_i)_{i \in \tau}$ a family of functions in $C(K,X)$ that is equivalent to the usual unit-vector basis of $c_0(\tau)$, and $(\psi_i)_{i \in \tau}$ a weak$^*$-null family in $C(K,X)^*$ such that
$$\psi_i(f_j) = \delta_{ij}, \forall i,j \in \tau.$$
For each $i \in \tau$, by the Riesz-Singer Representation Theorem \cite[Theorem 1.7.1]{C} there exists $\mu_i \in \text{rcabv}(K,X^*)$ such that $\|\mu_i\| = \|\psi_i\|$ and
$$\psi_i(f) = \int_K f \;d\mu_i, \forall f \in C(K,X).$$
Since $(\psi_i)_{i \in \tau}$ is weak$^*$-null in $C(K,X)^*$, by the Uniform Boundedness Principle we have
$$M = \sup_{i \in \tau} \|\psi_i\| < +\infty.$$

Let $D \subset K$ and $\mathcal{C} \subset C(K)$ be dense subsets of $K$ and $C(K)$ respectively with $|D| = \text{dens}(K)$ and $|\mathcal{C}|=\text{dens}(C(K))$. For each $i \in \tau$, by Lemma \ref{lemma1} there exist ${m_i} \geq 1$, $f^i_1, \ldots, f^i_{m_i} \in \mathcal{C}$ and $d^i_1, \ldots, d^i_{m_i} \in D$ satisfying
$$\|f_i - h_i\|_\infty < \frac{1}{2M},$$
where
$$h_i = \sum_{j=1}^{m_i} f^i_j (\cdot) f_i(d^i_j).$$
Therefore, for each $i \in\tau$ we obtain
$$1 = \psi_i(f_i) \leq |\psi_i(f_i-h_i)| + |\psi_i(h_i)| < \frac{1}{2} + |\psi_i(h_i)|,$$
and thus
$$\frac{1}{2} < |\psi_i(h_i)| = \left| \int_K h_i\; d\mu_i\right| \leq  \sum_{j=1}^{m_i} \left|\left( \int_K f^i_j\; d\mu_i\right) (f_i(d^i_j))\right|.$$

Since $K$ is infinite, by \cite[Proposition 7.6.5]{Se} we have 
\begin{equation}\label{desigualdadeSemadeni}
\aleph_0 \leq \text{dens}(K) \leq \text{w}(K) = \text{dens}(C(K)),
\end{equation}
and hence, by hypothesis,
\begin{equation}\label{desigualdadeteo2}
\text{cf}(\tau) > \text{w}(K) = \max( |D|, |\mathcal{C}|) \geq \aleph_0.
\end{equation}

Let $\mathcal{M} = \{m_i : i \in \tau\} \subset \mathbb{N}$ and for each $m \in \mathcal{M}$, put $\alpha_m = \{ i \in \tau : m_i = m\}$. By \eqref{desigualdadeteo2} and Lemma \ref{lemma2} there exists $n_0 \in \mathcal{M}$ such that $|\alpha_{n_0}| = \tau$. Setting $\tau_1 = \alpha_{n_0}$, we have
$$\sum_{j=1}^{n_0} \left| \left(\int_K f^i_j\; d\mu_i\right) (f_i(d^i_j))\right| > \frac{1}{2}, \forall i \in \tau_1.$$

Next, for each $j \in \{1, \ldots, n_0\}$ consider
$$\beta_j = \left\{i \in \tau_1 : \left|\left(\int_K f^i_j\; d\mu_i\right)(f_i(d^i_j))\right| > \frac{1}{2n_0}\right\}.$$
Since $\tau_1$ is infinite, again by Lemma \ref{lemma2} there exists $j_0 \in \{1, \ldots, n_0\}$ such that $|\beta_{j_0}| = |\tau_1| = \tau$. Setting $\tau_2 = \beta_{j_0}$, we obtain
$$\left| \left(\int_K f^i_{j_0}\; d\mu_i\right) (f_i(d^i_{j_0}))\right| > \frac{1}{2n_0}, \forall i \in \tau_2.$$

Let $\mathcal{F} = \{f^i_{j_0} : i \in \tau_2\} \subset \mathcal{C}$ and for each $f \in \mathcal{F}$, define $\gamma_f = \{i \in \tau_2 : f^i_{j_0} = f\}$.
By \eqref{desigualdadeteo2} and Lemma \ref{lemma2} there exists $g_0 \in \mathcal{F}$ such that $|\gamma_{g_0}| = |\tau_2| = \tau$. Setting $\tau_3 = \gamma_{g_0}$, we have
$$\left| \left(\int_K g_0\; d\mu_i\right) (f_i(d^i_{j_0}))\right| > \frac{1}{2n_0}, \forall i \in \tau_3.$$

Let $\mathcal{D} = \{d^i_{j_0} : i \in \tau_3\} \subset D$ and for each $d \in \mathcal{D}$, put $\kappa_d = \{i \in \tau_3 : d^i_{j_0} = d\}$.
By \eqref{desigualdadeteo2} and Lemma \ref{lemma2} there exists $d_0 \in \mathcal{D}$ such that $|\kappa_{d_0}| = |\tau_3| = \tau$. Finally, setting $\tau_4 = \kappa_{d_0}$, we obtain
$$|\varphi_i (f_i(d_0))| > \frac{1}{2n_0},$$
for each $i \in \tau_4$, where $\varphi_i = \int_K g_0\; d\mu_i \in X^*$.

By Theorem \ref{schlumprechtgeneralizado}, in order to complete the proof it suffices to show that there exists $\tau_5 \subset \tau_4$ such that $|\tau_5| = \tau$, $(f_i(d_0))_{i \in \tau_5}$ is equivalent to the unit-vector basis of $c_0(\tau_5)$ and $(\varphi_i)_{i \in \tau_5}$ is weak$^*$-null in $X^*$.

Given $x \in X$, observe that
$$\varphi_i(x) = \left( \int_K g_0\; d\mu_i\right) (x) = \int_K g_0(\cdot) x \;d\mu_i = \psi_i(g_0(\cdot)x), \forall i \in \tau_4.$$
Therefore,
$$(\varphi_i(x))_{i \in \tau_4} = (\psi_i(g_0(\cdot)x))_{i \in \tau_4} \in c_0(\tau_4),$$
since $(\psi_i)_{i \in \tau}$ is weak$^*$-null in $C(K,X)^*$ by hypothesis. Thus, $(\varphi_i)_{i \in \tau_4}$ is weak$^*$-null in $X^*$. Note also that
$$0 < \|\varphi_i\| \leq \|\psi_i\|\|g_0\|_\infty \leq M \|g_0\|_\infty, \forall i \in \tau_4.$$

Let $T : c_0(\tau) \to C(K,X)$ be an isomorphism from $c_0(\tau)$ onto its image such that $T(e_i)= f_i$, for each $i \in \tau$. Consider $S : C(K,X) \to X$ the bounded linear operator defined by
$$S(f) = f(d_0), \forall f \in C(K,X).$$
Notice  that
$$\|(S \circ T)(e_i)\| = \|f_i(d_0)\| \geq \frac{1}{2 Mn_0 \|g_0\|_\infty} >0, \forall i \in \tau_4.$$
Therefore, by Theorem \ref{rosenthal}, there exists $\tau_5 \subset \tau_4$ such that $|\tau_5|=|\tau_4| = \tau$ and $S \circ T_{|c_0(\tau_5)}$ is an isomophism onto its image; hence,
$$(f_i(d_0))_{i \in \tau_5} = ( S(T(e_i))_{i \in \tau_5}$$
is equivalent to the unit-vector basis of $c_0(\tau_5)$. Thus, by Theorem \ref{schlumprechtgeneralizado} we have $c_0(\tau) \stackrel{c}{\hookrightarrow} X$.

Case 2: $K$ is finite. Let $n = |K|$. It is easy to see that $C(K,X) = (X^n, \|\cdot\|_\infty)$ in this case.
By finite induction, we may assume $n =2$. By Theorem \ref{schlumprechtgeneralizado}, there exist a family $(x_i, y_i)_{i \in \tau}$ equivalent to the usual unit-vector basis of $c_0(\tau)$ in $X^2$ and $(x^*_i, y^*_i)_{i \in \tau}$ a weak$^*$-null family in $(X^*)^2$ such that
$$x^*_i(x_j) + y^*_i(y_j) = \delta_{ij}, \forall i,j \in \tau.$$
In particular, we have
$$|x^*_i(x_i)| + |y^*_i(y_i)| \geq |x^*_i(x_i) + y^*_i(y_i)| = 1, \forall i \in \tau.$$
Consider $\Gamma_1 = \{ i \in \tau : |x^*_i(x_i)| \geq 1/2\}$ and $\Gamma_2 = \{ i \in \tau : |y^*_i(y_i)| \geq 1/2\}$. Since $\tau$ is infinite, we have $|\Gamma_1| = \tau$ or $|\Gamma_2| = \tau$; without loss of generality, we may assume the former. By definition,
$$|x^*_i(x_i)| \geq \frac{1}{2}, \forall i \in \Gamma_1.$$

Given $x \in X$, observe that
$$(x^*_i(x))_{i \in \tau} = ((x^*_i,y^*_i)(x,0))_{i \in \tau} \in c_0(\tau),$$
since $(x^*_i, y^*_i)_{i \in \tau}$ is weak$^*$-null by hypothesis. Thus, $(x^*_i)_{i \in \tau}$ is weak$^*$-null in $X^*$, and therefore, by the Uniform Boundedness Principle we have
$$0<\delta = \sup_{i \in \tau} \|\varphi_i\|< +\infty.$$
Let $S : c_0(\tau) \to X^2$ be an isomorphism onto its image such that $S(e_i) = (x_i,y_i)$, for each $i \in \tau$. Consider $Q : X^2 \to X$ the bounded linear operator defined by
$$Q(x,y) = x, \forall x,y \in X.$$
Observe that
$$\|(Q \circ S)(e_i)\| = \|x_i\| \geq \frac{1}{2\delta}>0, \forall i \in \tau_1.$$
Therefore, by Theorem \ref{rosenthal}, there exists $\Gamma \subset \Gamma_1$ such that $|\Gamma| = |\Gamma_1| = \tau$ and $Q \circ S_{|c_0(\Gamma)}$ is an isomophism onto its image; hence,
$$(x_i)_{i \in \Gamma} = (Q(S(e_i))_{i \in \Gamma}$$
is equivalent to the unit-vector basis of $c_0(\Gamma)$. Thus, an appeal to Theorem \ref{schlumprechtgeneralizado} finishes the proof.

\section{Theorem \ref{teoprincipal} is optimal for every infinite cardinal $\tau$}

In this last section we show that  the assertion of the Theorem \ref{teoprincipal} does not remain true, in general,  when cf$(\tau) = \text{w}(K)$. We distinguish three cases: $\tau = \aleph_0$, regular uncountable $\tau$, and singular uncountable $\tau$.

Initially, assume  that  $\tau= \aleph_{0}$.  According to the Cembranos-Freniche's theorem, for all infinite compact metric space $K$ we have
$$c_{0} \stackrel{c}{\hookrightarrow} C(K, l_{\infty}).$$
So, despite cf$(\aleph_{0})=\aleph_{0}=\text{w}(K)$, it is well-known  that $l_{\infty}$ contains no complemented copies of $c_{0}$ \cite[Corollary 11, p. 156]{D}.

Now suppose that $\tau$ is regular and  $\tau> \aleph_{0}$. Let ${\mathfrak{m}}$ be an infinite cardinal and denote by ${\bf 2}^{\mathfrak{m}}=\{0, 1\}^{\mathfrak{m}}$ the Cantor cube. First of all, we need to  observe that there is a misprint in the statement of \cite[Theorem 5.1]{Ga}. Indeed,  recall that   Lindenstrauss proved that $L_{1}([0, 1]^{\mathfrak{m}})$ contains a copy of $l_{2}({\mathfrak{m}})$ \cite[Theorem 2.13]{Lacey}, and not a copy of $l_{2}(2^{\mathfrak{m}})$ as it is printed into the proof of \cite[Theorem 5.1]{Ga}.  So, 
following step by step the proof of  \cite[Theorem 5.1]{Ga} we see that the true statement of that theorem is:  

\begin{thm} \label{T} Let
 $\mathfrak{m}$  a cardinal satisfying 
$\aleph_0 < \emph{cf}(\aleph_\alpha) \leq \aleph_\alpha \leq \mathfrak{m},$
for some ordinal $\alpha$. If  $X$ is a Banach space, then
$$ c_0(\aleph_\alpha) \hookrightarrow X \iff c_0(\aleph_\alpha) \stackrel{c}{\hookrightarrow} C(\emph{\textbf{2}}^\mathfrak{m},X).$$
\end{thm}

Therefore by applying Theorem \ref{T} with $\aleph_\alpha = \mathfrak{m} = \tau$, we deduce that 
$$c_0(\tau) \stackrel{c}{\hookrightarrow} C(\textbf{2}^\tau, l_{\infty}(\tau)).$$
Thus, although cf$(\tau)=\tau=\text{w}(\textbf{2}^\tau)$ \cite[Corollary 8.2.7]{Se}, again by  \cite[Corollary 11, p. 156]{D} we know that $l_{\infty}(\tau)$ contains no complemented copies of $c_{0}(\tau)$.

Finally, assume that $\tau$ is singular. In this case, we need to recall some notation. If $\Gamma$ is a non-empty set and $X$ is a Banach space, we denote by $c_0(\Gamma,X)$ the Banach space of all $X$-valued maps $f$ on $\Gamma$ such that $(\|f(i)\|)_{i \in \Gamma} \in c_0(\Gamma)$, endowed with the supremum norm. If $I$ is an infinite set, we denote by $\gamma I$ its Alexandroff compactification. 

 Let $\lambda = \text{cf}(\tau)$ and consider $(\alpha_i)_{i \in \lambda}$ a strictly increasing family of cardinals satisfying $\alpha_i < \tau$, for each $i \in \lambda$, and $\sup_{i \in \lambda} \alpha_i = \tau$. Let $X$ be the Banach space of all families $(x_i)_{i \in \lambda}$ such that $x_i \in c_0(\alpha_i)$, for each $i \in \lambda$, and $\sum_{i \in \lambda} \|x_i\|_\infty < +\infty$, equipped with the norm
$$\|(x_i)_{i \in \lambda}\|_1 = \sum_{i \in \lambda} \|x_i\|_\infty.$$
It is not difficult to check that
$$c_0(\tau)\stackrel{c}{\hookrightarrow} c_0(\lambda, X).$$
Then, by \cite[Corollary 21.5.2]{Se} we conclude that
$$c_0(\tau)\stackrel{c}{\hookrightarrow} C(\gamma \lambda, X).$$
On the other hand, since $\alpha_i < \tau$ for every $i \in \lambda$, it follows that  
$c_0(\tau)$ is isomorphic to no subspace of  $c_0(\alpha_i)$. So,   by a standard gliding humps argument we can prove that $X$ contains no copies of $c_0(\tau)$, see for instance \cite{B}. However cf$(\tau) = \lambda = \text{w}(\gamma \lambda)$ and thus   we are done.

\end{document}